\documentclass[12pt]{amsart}
\usepackage{amssymb,latexsym,eufrak,amsmath,amscd, graphicx}
\usepackage[colorlinks=true,pagebackref,hyperindex]{hyperref}

\usepackage{amsfonts}
\usepackage{amscd}

\renewcommand{\phi}{\varphi}
\renewcommand{\epsilon}{\varepsilon}
\renewcommand{\theta}{\vartheta}

\def\ZZ{{\mathbf Z}}
\def\NN{{\mathbf N}}
\def\FF{{\mathbf F}}

\def\AAA{{\mathbf A}}
\def\RR{{\mathbf R}}
\def\QQ{{\mathbf Q}}

\def\cO{\mathcal{O}}

\def\fra{\mathfrak{a}}

\def\alb{{\alpha\beta}}
\DeclareMathOperator{\codim}{codim} 
\DeclareMathOperator{\Hom}{Hom}

 \DeclareMathOperator{\Spec}{Spec}
 
 \DeclareMathOperator{\lct}{lct}
 
 \DeclareMathOperator{\ord}{ord}
 \DeclareMathOperator{\fpt}{fpt}
\DeclareMathOperator{\cont}{Cont}
\DeclareMathOperator{\Sch}{Sch}
\DeclareMathOperator{\id}{id}

\newcommand{\llb}{[\negthinspace[}
\newcommand{\rrb}{]\negthinspace]}

\newtheorem{lemma}{Lemma}[section]
\newtheorem{theorem}[lemma]{Theorem}
\newtheorem{corollary}[lemma]{Corollary}
\newtheorem{proposition}[lemma]{Proposition}

\newtheorem{theoremalpha}{Theorem}

\theoremstyle{definition}
\newtheorem{definition}[lemma]{Definition}
\newtheorem{remark}[lemma]{Remark}

\theoremstyle{remark}
\newtheorem*{remark*}{Remark}
\newtheorem*{note*}{Note}

\begin{document}
\title{Log canonical thresholds in positive characteristic}

%\thanks{2000\,\emph{Mathematics Subject Classification}.
% Primary 13A35; Secondary 14B05, 32S40.
%\newline The author
%  was partially supported by the NSF grant DMS 0500127 and
  %by a Packard Fellowship}
%\keywords{Bernstein-Sato polynomials, V-filtration, test ideals}

\author[Z.Zhu]{Zhixian~Zhu}
\address{Department of Mathematics, University of Michigan,
Ann Arbor, MI 48109, USA} \email{{\tt zhixian@umich.edu}}
\maketitle

\markboth{Z.~Zhu}{Log canonical thresholds in positive characteristic}

\begin{abstract}
In this paper, we study the singularities of a pair $(X,Y)$ in arbitrary characteristic via jet schemes. For a smooth variety $X$ in characteristic $0$, Ein, Lazarsfeld and Musta\c{t}\v{a} showed that there is a correspondence between irreducible closed cylinders and divisorial valuations on $X$. Via this correspondence, one can relate the codimension of a cylinder to the log discrepancy of the corresponding divisorial valuation. We now extend this result to positive characteristic. In particular, we prove Musta\c{t}\v{a}'s log canonical threshold formula avoiding the use of log resolutions, making the formula available also in positive characteristic. As a consequence, we get a comparison theorem via reduction modulo $p$ and a version of Inversion of Adjunction in positive characteristic.
\end{abstract}

\section*{introduction}
Let $k$ be an algebraically closed field of arbitrary characteristic. Given $m\geq 0$ and a scheme $X$ over $k$, we denote by
$$X_m=\Hom(\Spec k[t]/(t^{m+1}),X)$$
the $m^{\text{th}}$ order jet scheme of $X$. If $X$ is a smooth variety of dimension $n$, then $X_m$ is a smooth variety of dimension $n(m+1)$ and the truncation morphism $\rho^{m+1}_{m}: X_{m+1}\rightarrow X_m$ is locally trivial with fiber $\AAA^n$.

The space of arcs $X_{\infty}$ is the projective limit of the $X_m$ and thus parameterizes all formal arcs on $X$. One writes $\psi_{m}: X_{\infty}\rightarrow X_m$ for the natural map. The inverse images of constructible subsets by the canonical projections $\psi_m: X_{\infty}\to X_m$ are called \emph{cylinders}. Interesting examples of such subsets arise as follows.
Consider a non-zero ideal sheaf $\fra\subset \mathcal{O}_{X}$ defining a subscheme $Y\subset X$. For every $p\geq 0$, the \emph{contact locus} of order $p$ of $\fra$  is a locally closed cylinder
$$\text{Cont}^p(Y)=\text{Cont}^p(\fra):=\left\{\gamma\in X_{\infty}~|~\ord_{\gamma}(\fra)=p\right\}.$$
Similarly, we define a closed cylinder
$$\text{Cont}^{\geq p}(Y)=\text{Cont}^{\geq p}(\fra):=\left\{\gamma\in X_{\infty}~|~\ord_{\gamma}(\fra)\geq p\right\}.$$
Jet schemes and arc spaces are fundamental objects for the theory of motivic integration, due to Kontsevich
\cite{Kon} and Denef and Loeser \cite{DL}. Furthermore, in characteristic $0$, using the central result of this theory, the Change of Variable formula, one can show that there is a close link between the log discrepancy defined in terms of divisorial valuations and the geometry of the contact loci in arc spaces. This link was first explored by Musta\c{t}\v{a} in \cite{Mus1} and \cite{Mmus2}, and then further studied in \cite{EMY}, \cite{Ish}, \cite{ELM} and \cite{FEI}.

The main purpose of this paper is to show that the correspondence between irreducible closed cylinders and divisorial valuations in \cite{ELM} holds for smooth varieties of arbitrary characteristic. We now explain it as follows. Let $X$ be a smooth variety of dimension $n$ over $k$. An important class of valuations of the function field $k(X)$ of $X$ consists of \emph{divisorial valuations}. These are the valuations of the form
$$\nu=q\cdot\ord_E: k(X)^*\rightarrow \ZZ$$
where $E$ is a divisor over $X$, (that is, a prime divisor on a normal variety $X'$, having a birational morphism to $X$) and $q$ is a positive integer number. One can associate an integer number to a divisorial valuation $\nu=q\cdot\ord_E$, called the log discrepancy of $\nu$, equal to $q\cdot(1+\ord_{E}(K_{X'/X}))$, where $K_{X'/X}$ is the relative canonical divisor. These numbers determine the log canonical threshold $\lct(X,Y)$ of a pair $(X,Y)$, where $Y$ is a closed subscheme of $X$.

For every closed irreducible nonempty cylinder $C\subset X_{\infty}$ which does not dominate $X$, one defines
$$\ord_C:k(X)^*\rightarrow \ZZ$$
by taking the order of vanishing along the generic point of $C$. These valuations are called \emph{cylinder valuations}. If $C$ is an irreducible component of $\text{Cont}^{\geq p}(Y)$ for some subscheme $Y$ of $X$, the valuation $\ord_C$ is called a \emph{contact valuation}. It is easy to see that every divisorial valuation is a cylinder valuation. When the ground field is of characteristic zero, Ein, Lazarsfeld and Musta\c{t}\v{a} showed the above classes of valuations coincide, by showing that:
\begin{enumerate}
\item[(a)] Every contact valuation is a divisorial valuation (\cite[Theorem A]{ELM});\\
\item[(b)] every cylinder valuation is a contact valuation (\cite[Theorem C]{ELM}).
\end{enumerate}
We thus have a correspondence between the irreducible closed cylinders that do not dominate $X$ and divisorial valuations. Via this correspondence, one can relate the codimension of the cylinder to the log discrepancy of the divisorial valuation. This yields a quick proof of Musta\c{t}\v{a}'s log canonical threshold formula.

\begin{theorem} ${\rm (}$\cite{Mmus2}{\rm [Corollary 3.6]}, \cite{ELM}{\rm[Corollary B]}${\rm )}$
If $X$ is a smooth complex variety and $Y\subset X$ is a closed subscheme, then the log canonical threshold of the pair $(X,Y)$ is given by
\begin{equation*}
\lct(X,Y)=\min\limits_{m}\left\{\frac{\codim(Y_m,X_m)}{m+1}\right\}.
\end{equation*}
\end{theorem}
The key ingredients in the proofs of the above theorems in \cite{ELM} are the Change of Variable formula developed in the theory of motivic integration and the existence of log resolutions. While a version of the Change of Variable formula also holds in positive characteristic, the use of log resolutions in the proofs in \cite{Mmus2} and \cite{ELM} restricted the result to ground fields of characteristic zero. In this paper, we show by induction on the codimension of cylinders and only using the Change of Variable formula for blow-ups along smooth centers that the above correspondence between divisorial valuations and cylinders holds in arbitrary characteristic.
\begin{theoremalpha}\label{corr}
Let $X$ be a smooth variety of dimension $n$ over a perfect field $k$. There is a correspondence between irreducible closed cylinders $C\subset X_{\infty}$ that do not dominate $X$ and divisorial valuations as follows:
\begin{itemize}
\item[(1)]
If $C$ is an irreducible closed cylinder which does not dominate $X$, then there is a divisor $E$ over $X$ and a positive integer $q$ such that
\begin{equation*}
\ord_C=q\cdot\ord_E.
\end{equation*}
Furthermore, we have $\codim C\geq q\cdot(1+\ord_E(K_{-/X}))$.\\
\item[(2)] To every divisor $E$ over $X$ and every positive integer $q$, we can associate an irreducible closed cylinder $C$ which does not dominate $X$ such that
\begin{equation*}
\ord_C=q\cdot\ord_E {~and}~\codim C=q\cdot(1+\ord_E(K_{-/X})).
\end{equation*}
\end{itemize}
\end{theoremalpha}
We thus are able to prove the log canonical threshold formula avoiding use the log resolutions.
\begin{theoremalpha}\label{jetformula}
Let $X$ be a smooth variety of dimension $n$ defined over a perfect field $k$, and $Y$ be a closed subscheme. Then
$$\lct(X, Y)=\inf\limits_{C\subset X_{\infty}}\frac{\codim C}{\ord_C(Y)}=\inf\limits_{m\geq 0}\frac{\codim(Y_m,X_m)}{m+1}$$
where $C$ varies over the irreducible closed cylinders which do not dominate $X$.
\end{theoremalpha}

The paper is organized as follows. In the first section, we review some basic definitions and notation concerning cylinders and valuations. In section $2$ we prove a version of the Change of Variable formula and construct the correspondence between cylinders and divisorial valuations. The proofs of Theorem \ref{corr} and Theorem \ref{jetformula} are also given in $\S 2$.  In section $3$, we apply the log canonical threshold formula and obtain a comparison theorem via reduction modulo $p$, as well as a version of Inversion of Adjunction in positive characteristic.

\section*{Acknowledgement}
I am grateful to my advisor Mircea Musta\c{t}\v{a} for suggesting this project to me and useful discussions.

\section{introduction to jet schemes and log canonical threshold}
In this section, we first recall the definition and basic properties of jet schemes and arc spaces. For a more detailed discussion of jet schemes, see \cite{EM} or \cite{Mus1}.

We start with the absolute setting and explain the relative version of jet schemes later. Let $k$ be a field of arbitrary characteristic. A variety is an integral scheme, separated and of finite type over $k$. Given a scheme $X$ of finite type over $k$ and an integer $m\geq 0$, the $m^{\text{th}}$ order jet scheme $X_m$  of $X$ is a scheme of finite type over $k$ satisfying the following adjunction
\begin{equation}\label{adjunction}
\Hom_{\Sch/k}(Y,X_m)\cong\Hom_{\Sch/k}(Y\times \Spec k[t]/(t^{m+1}), X)
\end{equation}
for every scheme $Y$ of  finite type over $k$. It follows that if $X_m$ exists, then it is unique up to a canonical isomorphism. We will show the existence in Proposition \ref{existofJS}.

Let $L$ be a field extension of $k$. A morphism $\Spec L[t]/(t^{m+1})\rightarrow X$ is called an $L$--valued $m$--jet of $X$. If $\gamma_m$ is a point in $X_m$, we call it an $m$--jet of $X$. If $\kappa$ is the residue field of $\gamma_m$, then $\gamma_m$ induces a morphism $(\gamma_m)_{\kappa}:\Spec \kappa[t]/(t^{m+1})\rightarrow X$.

It is easy to check that $X_0=X$. For every $j\leq m$, the natural ring homomorphism $k[t]/(t^{m+1})\rightarrow  k[t]/(t^{j+1})$ induces a closed embedding $$\Spec k[t]/(t^{j+1})\rightarrow \Spec k[t]/(t^{m+1})$$ and the adjunction (\ref{adjunction}) induces a truncation map $\rho^m_j: X_m\rightarrow X_j$. For simplicity, we usually write $\pi_m^X$ or simply  $\pi_m$ for the projection $\rho^m_0:X_m\rightarrow X=X_0$. A morphism of schemes $f: X\rightarrow Y$ induces morphisms $f_m \colon X_m\to Y_m$ for every $m$. At the level of $L$--valued points, this takes an $L[t]/(t^{m+1})$--valued point $\gamma$ of $X_m$ to $f\circ \gamma$.
For every point $x\in X$, we write $X_{m,x}$ for the fiber of $\pi_m$ at $x$, the $m$--jets of $X$ centered at $x$.

In section $4$, we will use the relative version of jet schemes. We now recall some basic facts about this context.

We work over a fixed separated scheme $S$ of finite type over a noetherian ring  $R$. Let $f: W\rightarrow S$ be a scheme of finite type over $S$. If $s$ is a point in $S$, we denote by $W_s$ the fiber of $f$ over $s$.

\begin{definition}The $m^{\text{th}}$ relative jet scheme $(W/S)_m$ satisfies the following adjunction
\begin{equation}\label{reladj}
\Hom_{\Sch/S}(Y\times_{R} \Spec R[t]/(t^{m+1}), W)\cong \Hom_{\Sch/S}(Y, (W/S)_m),
\end{equation}
for every scheme of finite type $Y$ over $S$.
\end{definition}
As in the absolute setting, we have $(W/S)_0\cong W$. If $(W/S)_m$ and $(W/S)_j$ exist with $m\geq j$, then there is a canonical projection $\rho_j^m: (W/S)_m\rightarrow (W/S)_j$. For simplicity, we usually write $\pi_{m}$ for the projection $\rho_0^m: (W/S)_m\rightarrow W$.

We now prove the existence of the relative jet schemes, which is similar to that of the absolute case. For details, see \cite{Mmus2}.

\begin{proposition}\label{existofJS}
If $f: W\rightarrow S$ is a scheme of finite type over $S$, the $m^{\textit{th}}$ order relative jet scheme $(W/S)_m$ exists for every $m\in \NN$.
\end{proposition}

It is clear that the construction of relative jet scheme is compatible with open embeddings.

\begin{lemma}\label{openembbeding}
Let $U$ be an open subset of a scheme $W$ over $S$. If $(W/S)_m$ exists, then $(U/S)_m$ exists and $(U/S)_m\cong(\pi_m^{W})^{-1}(U)$.
\end{lemma}

We now prove Proposition \ref{existofJS} by first constructing the relative jet scheme locally and gluing the schemes along overlaps.

\begin{proof} By covering $S$ by affine open subschemes, we may and will assume $S$ is an affine scheme. Let $S=\Spec A$, where $A$ is a finitely generated $R$--algebra. We first construct $(W/S)_m$ when $W$ is an affine scheme over $S$.  Let $W=\Spec B$ for some $A$--algebra $B$. Consider a closed embedding $W\rightarrow \AAA^n_S$ such that $W$ is defined by the ideal $I=(f_1, \ldots, f_r)\subseteq A[x_1,\ldots, x_n]$.  An $S$--morphism $\phi: \Spec B[t]/(t^{m+1})\rightarrow W$ is given by $\phi^*(x_i)=\sum\limits_{j=0}^m b_{i,j}t^j$ with $b_{i,j}\in B$ such that $f_l(\phi^*(x_1),\ldots, \phi^*(x_n))=0$ in $B[t]/(t^{m+1})$ for every $l$.

Given any $u_i=\sum\limits_{j=0}^m a_{i,j} t^j$ in $A[t]/(t^{m+1})$ for $1\leq i\leq n$, we can write
\begin{equation}\label{affinejetseqn}
f_l(u_1,\ldots,u_n)=\sum\limits_{p=0}^{m} g_{l,p}(a_{i,j})t^p,
\end{equation}
for some polynomials $g_{l,p}$ in $A[a_{i,j}]$ with $1\leq i\leq n$ and $0\leq j\leq m$. Let $Z$ be the closed subscheme of $\AAA^{n(m+1)}_S=\Spec A[a_{i,j}]$ defined by $(g_{l,p})$ for $1\leq l\leq r$ and $0\leq p\leq m$. It is clear that $\phi$ is a $B[t]/(t^{m+1})$--valued point of $(W/S)$ if and only if the corresponding $(b_{i,j})$ defines a $B$--valued point of $Z$. Hence $(X/S)_m\cong Z$.

Given $W$ an arbitrary $S$--scheme of finite type, we consider an affine open cover $W=\bigcup\limits_{\alpha} W_{\alpha}$. We have seen that $(W_{\alpha}/S)_m$ exists for every $m\geq 0$. Let $\pi_m^{\alpha}: (W_{\alpha}/S)_m\rightarrow W_{\alpha}$ be the canonical projection. For every $\alpha$ and $\beta$, we write $W_{\alb}=W_{\alpha}\cap W_{\beta}$. The inverse image $(\pi^\alpha_m)^{-1}(W_{\alb})$ and $(\pi^\beta_m)^{-1}(W_{\alb})$ are canonically isomorphic since they are isomorphic to $(W_{\alb}/S)_m$. Hence we can construct a scheme $(W/S)_m$ by gluing the schemes $(W_{\alpha}/S)_m$ along their overlaps. Moreover, the projections $\pi_m^{\alpha}$ glue to give an $S$--morphism $$\pi_m: (W/S)_m\rightarrow W.$$ It is clear that $(W/S)_m$ is the $m^{\textit{th}}$ relative jet scheme of $W$ over $S$.
\end{proof}

For every scheme morphism $S'\rightarrow S$ and every $W/S$ as above, we denote by $W'$ the fiber product $W\times_S S'$. For every point $s\in S$, we denote by $W_s$ the fiber of $W$ over $s$. By the functorial definition of relative jet schemes, we can check that $$(W'/S')_m\cong (W/S)_m\times_S S'$$ for every $m$. In particular, for every $s\in S$, we conclude that the fiber of $(W/S)_{m}\rightarrow S$ over $s$ is isomorphic to $(W_s)_m$.

Recall that $\pi_m: (W/S)_m\rightarrow W$ is the canonical projection. We now show that there is an $S$--morphism, called the zero-section map, $\sigma_m: W\rightarrow (W/S)_m$ such that $\pi_m\circ \sigma_m=\id_W$. We have a natural map $g_m: W\times \Spec R[t]/(t^{m+1})\rightarrow W$, the projection onto the first factor. By (\ref{reladj}), $g_m$ induces a morphism $\sigma_m^W: W\rightarrow (W/S)_m$, the {\it zero-section} of $\pi_m$. For simplicity, we usually write $\sigma_m$ for $\sigma_m^W$. One can check that $\pi_m\circ \sigma_m=\id_W$.

Note that for every $m$ and every scheme $W$ over $S$, there is a natural action:
$$\Gamma_m: \AAA^{1}_S\times_S (W/S)_m\rightarrow (W/S)_m$$ of the affine group $\AAA^1_S$ on the jet schemes $(W/S)_m$ defined as follows. For an $A$--valued point $(a,\gamma_m)$ of $\AAA^1_S\times_S (W/S)_m$ where $a\in A$ and $\gamma_m:\Spec A[t]/(t^{m+1})\rightarrow W$, we define $\Gamma_m(a,\gamma_m)$ as the composition map $\Spec A[t]/(t^{m+1})\xrightarrow{a^*}\Spec A[t]/(t^{m+1})\xrightarrow{\gamma_m} X$, where $a^*$ corresponds to the $A$--algebra homomorphism $A[t]/(t^{m+1})\rightarrow A[t]/(t^{m+1})$ mapping $t$ to $at$. One can check that the image of the zero section $\sigma_m$ is equal to $\Gamma_m(\{0\}\times (W/S)_m)$.

\begin{lemma}\label{semicont}
Let $f:W\rightarrow S$ be a family of schemes and $\tau: S\rightarrow W$ a section of $f$. For every $m\geq 1,$ the function
$$d(s)=\dim(\pi_m^{W_s})^{-1}(\tau(s))$$
is upper semi-continuous on $S$.
\end{lemma}

\begin{proof}
Due to the local nature of the assertion, we may assume that $S=\Spec A$ is an affine scheme. Given a point $s\in S$, we denote by $w=\tau(s)$ in $W$. Let $W'$ be an open affine neighborhood of $w$ in $W$.  Consider the restriction map $f': W'\rightarrow S$ of $f$,  one can show that there is an nonzero element $h\in A$ such that such that $\tau$ maps the affine neighborhood $S'\cong \Spec A_h$ of $s$ into $W'$. Let $W''$ be the affine neighborhood $(f')^{-1}(S')$ of $w$ and $f'':W''\rightarrow S'$ the induced map. The restriction of $\tau$ defines a section $\tau': S'\rightarrow W''$. Replacing $f$ by $f''$ and $\tau$ by $\tau'$, we may and will assume that both $W$ and $S$ are affine schemes. Let $W=\Spec B$, where $B$ is a finitely generated $A$--algebra. The section $\tau$ induce a ring homomorphism $\tau^*: B\twoheadrightarrow A$. Choose $A$--algebra generators $u_i$ of $B$ such that $\tau^*(u_i)=0$.
Let $C$ be the polynomial ring $A[x_1,\ldots, x_n]$. We define a ring homomorphism $\varphi:C\rightarrow B$ which maps $x_i$ to $u_i$ for every $i$. Let $I=(f_1,\ldots,f_r)$ be the kernel of $\varphi$. One can check that $f_l\in (x_1, \ldots, x_n)$ for every $l$ with $1\leq l\leq r$. Hence $W$ is a closed subscheme of $\AAA_S^n=\Spec A[x_1,\ldots, x_n]$ defined by the system of polynomials $(f_{l})$ and the zero section $o: S\rightarrow \AAA_S^n$ factors through $\tau$. It is clear that $(\AAA^n_S)_m=\Spec A[a_{i,j}]\cong \AAA^{n(m+1)}_S$ for $1\leq i\leq n$ and $0\leq j\leq m$ and $\sigma_m^{\AAA_S^n}\circ o: S\rightarrow \AAA^{n(m+1)}_S$ is the zero-section.

We thus obtain an embedding $(W/S)_m\subset\AAA_S^{(m+1)n}$ which induces an embedding $(\pi_m^W)^{-1}(\tau(S))\subset (\pi_m^{\AAA_S^n})^{-1}(o(S))\cong\AAA_S^{mn}=
%$such that $\sigma_m^W\circ \tau$ corresponding to the zero-section of $\AAA_S^{mn}=
\Spec A[a_{i,j}]$ for $1\leq i\leq n$ and $1\leq j\leq m$. Recall that $(W/S)_m$ as a subscheme of $\AAA^{n(m+1)}_S$ is defined by the polynomials $g_{l,p}$ in equation (\ref{affinejetseqn}). Let $\deg a_{i,j}=j$ for $1\leq i\leq n$ and $1\leq j\leq m$. Since $f_{l}$ has no constant terms, we can check that each $g_{l,p}$ is homogenous of degree $p$. We deduce that the coordinate ring of $(\pi_m^W)^{-1}(\tau(S))$, denoted by $T$, is isomorphic to $A[a_{i,j}]/(g_{l,p})$. Hence $T$ is a graded $A$--algebra.

For every $s\in S$ corresponding to a prime ideal $\mathfrak{p}$ of $A$, we obtain that
$$d(s)=\dim(\pi_m^{W_s})^{-1}(\tau(s))=\dim(T\otimes_{A} A/\mathfrak{p}).$$
Our assertion follows from a semi-continuity result on the dimension of fibers of a projective morphism (see \cite[Theorem 14.8]{Eis}).
\end{proof}

\begin{remark}\label{closed}
Let $X$ be a smooth variety over a field $k$ and $Y$ a closed subscheme of $X$. If $T$ is an irreducible component of $Y_m$ for some $m$, then $T$ is invariant under the action of $\AAA^1$. Since $\pi_m(T)=\sigma_m^{-1}(T\cap \sigma_m(X))$, it follows that $\pi_m(T)$ is closed in $X$.
\end{remark}

We now review some definitions in the theory of singularities of pairs $(X,Y)$. We refer the reader to \cite[Section 2.3]{KM} for a more detailed introduction. From now on, we always assume varieties are $\QQ$-Gorenstein. Suppose $X'$ is a normal variety over $k$ and $f: X' \rightarrow X$ is a birational (not necessarily proper) map. Let $E$ be a prime divisor on $X'$. Any such $E$ is called a divisor \emph {over} $X$. The local ring $\cO_{X',E}\subset k(X')$ is a DVR which corresponds to a divisorial valuation $\ord_{E}$ on $k(X)=k(X')$. The closure of $f(E)$ in $X$ is called the {\it center} of $E$, denoted by $c_X(E)$. If $f': X''\rightarrow X$ is another birational morphism and $F\subset X''$ is a prime divisor such that $\ord_E=\ord_F$ as valuations of $k(X)$, then we consider $E$ and $F$ to define the same divisor over $X$.

Let $E$ be a prime divisor over $X$ as above. If $Z$ is a closed subscheme of $X$, then we define $\ord_E(Z)$ as follows. We may assume that $E$ is a divisor on $X'$ and that the scheme-theoretic inverse image $f^{-1}(Z)$ is an effective Cartier divisor on $X'$. Then $\ord_E(Z)$ is the coefficient of $E$ in $f^{-1}(Z)$. Recall that the {\it relative canonical divisor} $K_{X'/X}$ is the unique $\QQ$--diviosr supported on the exceptional locus of $f$ such that $mK_{X'/X}$ is linearly equivalent with $mK_{X'}-f^*(mK_X)$ for some $m$. When $X$ is smooth, we can alternatively describe $K_{X'/X}$ as follows. Let $U$ be a smooth open subset of $X'$ such that $U\cap E\neq \emptyset$. The restriction of $f$ to $U$ is a birational morphism of smooth varieties, we denote it by $g$. In this case, the relative canonical divisor $K_{U/X}$ is the effective Cartier divisor defined by $\det(dg)$ on $U$.

We also define $\ord_E(K_{-/X})$ as the coefficient of $E$ in $K_{U/X}$. Note that both $\ord_E(Y)$ and $\ord_E(K_{-/X})$ do not depend on the particular choice of $f$, $X'$ and $U$.

For every real number $c>0$, the {\it log discrepancy} of the pair $(X,cY)$ with respect to $E$ is
\begin{equation*}
a(E;X,cY):=\ord_E(K_{-/X})+1-c\cdot\ord_E{Y}.
\end{equation*}

Let $Z$ be a closed subset of $X$. A pair $(X,cY)$ is {\it Kawamata log terminal} (klt for short) around $Z$ if $a(E;X,cY)>0$ for every divisor $E$ over $X$ such that $c_X(E)\cap Z\neq \emptyset$.

The log canonical threshold of $(X,Y)$ at $Z$, denoted by $\lct_Z(X, Y)$, is defined as follows: if $Y=X$, we set $\lct_Z(X,Y)=0$, otherwise
\begin{equation*}
\lct_Z(X,Y)=\sup\{c\in \RR_{\geq 0}~| ~(X, cY) ~\text{is klt around } Z\}.
\end{equation*}
In particular, $\lct_Z(X,Y)=\infty$ if and only if $Z\cap Y=\emptyset$. If $Z=X$, we simply write $\lct(X,Y)$ for $\lct_Z(X,Y)$.

By the definition of $a(E;X,Y)$, we obtain that
\begin{eqnarray*}
\lct_Z(X,Y)&=&\sup\left\{c\in \RR ~|~ c\cdot\ord_E(Y)< \ord_E(K_{-/X})+1 \text{~for ~every}~E\right\}\\
      &=&\inf\limits_{E/X}\frac{\ord_E(K_{-/X})+1}{\ord_E{Y}}
\end{eqnarray*}
where $E$ varies over all divisors over $X$ such that $c_X(E)\cap Z\neq \emptyset$.

\begin{remark}
The definition of log canonical threshold involves all exceptional divisors over $X$. In characteristic zero, it is enough to only consider the divisors on a log resolution, see \cite[Proposition 8.5]{Kol}. In particular, we deduce that $\lct_Z(X,Y)$ is a positive rational number. In positive characteristics, it is not clear that $\lct_Z(X,Y)>0$. We will see in $\S 3$ as a corollary of inversion of adjunction that we have, as in characteristic zero, $\lct_x(X,Y)\geq 1/{\ord_x(Y)}>0$, for every point $x\in Y$. Here $\ord_x(Y)$ is the maximal integer value $q$ such that the ideal $I_{Y,x}\subseteq m^q_{X,x}$.
\end{remark}

\section{Arc Spaces and Contact Loci}
We now turn to the projective limit of jet schemes. It follows from the description in the proof of Proposition \ref{existofJS} that the projective system
 \begin{equation*}
 \cdots \rightarrow X_m\rightarrow X_{m-1}\rightarrow \cdots \rightarrow X_0
 \end{equation*}
consists of affine morphisms. Hence the projective limit exists in the category of schemes over $k$. This is called the \emph {space of arcs} of $X$, denoted by $X_{\infty}$. Note that in general, it is not of finite type over $k$. There are natural projection morphisms $\psi_m: X_{\infty}\rightarrow X_m$. It follows from the projective limit definition and the functorial description of the jet schemes that for every $k$--field extension $A$ we have
\begin{equation*}
\Hom(\Spec(A), X_{\infty})\simeq
\underleftarrow\Hom(\Spec\,A[t]/(t^{m+1}), X)
\simeq\Hom(\Spec\, A\llb t\rrb,X)
\end{equation*}
In particular, for every field extension $L$ of $k$, an $L$--valued point of
$X_{\infty}$, called an $L${\it--valued arc}, corresponds to a morphism from $\Spec\,L\llb t\rrb$ to $X$. We denote the closed point of $\Spec L\llb t\rrb$ by $0$ and by $\eta$ the generic point. A point in $X_{\infty}$ is called an \emph{arc in $X$}. If $\gamma$ is a point in $X_{\infty}$ with residue field $\kappa$, $\gamma$ induces a $\kappa$--valued arc, i.e. a morphism $\gamma_{\kappa}: \Spec\,\kappa\llb t\rrb\rightarrow X$.
If $f\colon X\to Y$ is a morphism of schemes of finite type, by taking the
projective limit of the morphisms $f_m:X_m\rightarrow Y_m$ we get a morphism
$f_{\infty}\colon X_{\infty}\to Y_{\infty}$.

For every scheme $X$, a \emph{cylinder} in $X_{\infty}$ is a subset of the form $C=\psi_m^{-1}(S)$, for some $m$ and some
constructible subset $S\subseteq X_m$. If $X$ is a smooth variety of pure dimension $n$ over $k$, then all truncation maps $\rho^m_{m-1}$ are locally trivial with fiber $\AAA^n$. In particular, all projections $\psi_m: X_{\infty}\rightarrow X_m$ are surjective and $\dim X_m=(m+1)n$.

From now on, we will assume that $X$ is smooth and of pure dimension $n$. We say that a cylinder $C=\psi_m^{-1}(S)$ is \emph{irreducible (closed, open, locally closed)} if so is $S$. It is clear that all these properties of $C$ do not depend on the particular choice of $m$ and $S$. We define the {\it codimension of} $C$ by
$$\codim C: =\codim(S,X_m)=(m+1)n-\dim S.$$
Since the truncation maps are locally trivial, $\codim C$ is independent of the particular choice of $m$ and $S$.

For a closed subscheme $Z$ of a scheme $X$ defined by the ideal sheaf $\fra$ and for an $L$--valued arc $\gamma: \Spec L \llb t \rrb\rightarrow X$, the inverse image of $Z$ by $\gamma$ is defined by a principal ideal in $L \llb t \rrb$. If this ideal is generated by $t^e$ with $e\geq 0$, then we define the vanishing order of $\gamma$ along $Z$ to be $\ord_{\gamma}(Z)=e$. On the other hand, if this is the zero ideal, we put $\ord_{\gamma}(Z)=\infty$. If $\gamma$ is a point in $X_{\infty}$ with residue field $L$, then we define $\ord_{\gamma}(Z)$ by considering the corresponding morphism $\Spec L \llb t\rrb \rightarrow X$.
The \emph {contact locus of order} $e$ with $Z$ is the subset of $X_{\infty}$
$${\rm Cont}^e(Z)={\rm Cont}^e(\fra):=\{\gamma\in
X_{\infty}\mid\ord_{\gamma}(Z)=e\}.$$
We similarly define $$\cont^{\geq e}(Z)={\rm Cont}^{\geq e}(\fra):=\{\gamma\in
X_{\infty}\mid\ord_{\gamma}(Z)\geq e \}.$$
For $m\geq e$, we can define constructible subsets ${\rm Cont}^e(Z)_m$ and ${\rm Cont}^{\geq e}(Z)_m$ of $X_m$ in the obvious way. (In fact, the former one is locally closed, while the latter one is closed.) By definition, we have
\begin{equation*}
{\rm Cont}^e(Z)=\psi^{-1}_m({\rm Cont}^e(Z)_m) ~\text{and}~ {\rm Cont}^{\geq e}(Z)=\psi^{-1}_m({\rm Cont}^{\geq e}(Z)_m).
\end{equation*}
This implies that ${\rm Cont}^{\geq e}(Z)$ is a closed cylinder and ${\rm Cont}^e(Z)$  is a locally closed cylinder in $X_\infty$.

\section{cylinder valuations and divisorial valuations}
The main goal of this section is to establish the correspondence between cylinders and divisorial valuations described in the introduction. Let $X$ be a variety over a field $k$. Recall that a subset $C$ of $X_{\infty}$ is {\it thin} if there is a proper closed subscheme $Z$ of $X$ such that $C\subset Z_{\infty}$.
\begin{lemma}\label{notthin}
Let $X$ be a smooth variety over $k$. If $C$ is a nonempty cylinder in $X_{\infty}$, then $C$ is not thin.
\end{lemma}
For the proof of Lemma \ref{notthin}, see \cite{ELM}[Proposition 1].

\begin{lemma}\label{proper lift}
Let $f: X'\rightarrow X$ be a proper birational morphism of schemes over $k$. Let $Z$ be a closed subset of $X$ and $F=f^{-1}(Z)$. If $f$ is an isomorphism over $X\setminus Z$, then the restriction map of $f_{\infty}$
$$\phi: X'_{\infty}\setminus F_{\infty}\rightarrow X_{\infty}\setminus Z_{\infty}$$
is bijective on the $L$--valued points for every field extension $L$ of $k$. In particular, $\phi$ is surjective.
\end{lemma}

\begin{proof}
Since $f$ is proper, the Valuative Criterion for properness implies that an arc $\gamma: \Spec L\llb t\rrb\rightarrow X$ lies in the image of $f_{\infty}$ if and only if the induced morphism $\gamma_{\eta}: \Spec L(\!(t)\!)\rightarrow X$ can be lifted to $X'$. $\gamma$ is not contained in $Z_{\infty}$ implies that $\gamma_{\eta}$ factor through $X\setminus Z\hookrightarrow X$. Since $f$ is an isomorphism over $X\setminus Z$, hence there is a unique lifting of $\gamma_{\eta}$ to $X'$. This shows that $\phi$ is surjective. The injectivity of $\phi$ follows from the Valuative Criterion for separatedness of $f$. The last assertion follows from the fact that a morphism of schemes (not necessary to be of finite type) over $k$ is surjective if the induced map on $L$--valued points is surjective for every field extension $L$.
\end{proof}

The Change of Variable Theorem due to Kontsevish \cite{Kon} and Denef and Loeser \cite{DL} will play an important role in our arguments. We now state a special case of this theorem as Lemma \ref{change}.

\begin{lemma}\label{change}
Let $X$ be a smooth variety of dimension $n$ over $k$ and $Z$ a smooth irreducible subvariety of codimension $c\geq 2$. Let $f: X'\rightarrow X$ be the blow up of $X$ along $Z$ and $E$ the exceptional divisor.
\begin{enumerate}
\item[(a)] For every positive integer $e$ and every $m\geq 2e$, the induced morphism $$\psi^{X'}_{m}(\cont^e(K_{X'/X}))\rightarrow f_m(\psi^{X'}_{m}(\cont^e(K_{X'/X})))$$ is a piecewise trivial $\AAA^e$ fibration.
\item[(b)] For every $m\geq 2e$, the fiber of $f_m$ over a point $\gamma_m\in f_m(\psi^{X'}_{m}({\cont}^e(K_{X'/X})))$  is contained in a fiber of $X'_m\rightarrow X'_{m-e}$.
\end{enumerate}
\end{lemma}
For the proof of Lemma \ref{change}, see \cite[Theorem 3.3]{Bli}.

Let $X$ be a smooth variety of dimension $n$ over $k$. For every irreducible cylinder $C$ which does not dominate $X$, we define a discrete valuation as follows. Let $\gamma$ be the generic point of $C$ with residue field $L$.  We thus have an induced ring homomorphism $\gamma^*: \mathcal{O}_{X,\gamma(0)}\rightarrow \Spec L\llb t\rrb$. Lemma \ref{notthin} implies that $\ker\gamma^*$ is zero. Hence $\gamma^*$ extends to an injective homomorphism $\gamma^*: k(X)\rightarrow L(\!(t)\!)$.
 We define a map
 $$\ord_C: k(X)^*\rightarrow \ZZ$$ by $\ord_C(f) := \ord_{\gamma}(f)=\ord_t(\gamma^*(f))$. If $C$ does not dominate $X$, then $\ord_C$ is a discrete valuation. If $C'$ is a dense subcylinder of $C$, then they define the same valuation. Given an element $f\in k(X)^*$, we can check that $\ord_{C}(f)=\ord_{\gamma'}(f)$ for general point $\gamma'$ in $C$.

From now on we assume that $k$ is a perfect field. We first prove that every valuation defined by a cylinder is a divisorial valuation.
\begin{lemma}\label{cyltoval}
 If $C$ is an irreducible closed cylinder in $X_{\infty}$ which does not dominate $X$, then there exist a divisor $E$ over $X$ and a positive integer $q$ such that
\begin{equation}\label{ctoe}
\ord_C=q\cdot \ord_E.
\end{equation}
Furthermore, we have $\codim(C)\geq q\cdot(1+\ord_E(K_{-/X}))$.
\end{lemma}

\begin{proof}
We will prove that such divisor $E$ can be reached by a sequence of blow ups of smooth centers after shrinking to suitable open subsets. Let $(R,m)$ be the valuation ring associated to the valuation $\ord_C$. Suppose that $C$ is $\psi^{-1}_m(S)$ for some closed irreducible subset $S$ in $X_m$. Chevalley's Theorem implies that the image of the cylinder $C$ by the projection $\psi_0(C)=\pi_{m}(S)$ is a constructible set. We denote its closure in $X$ by $Z$. This is \emph{the center of} $\ord_C$. If $C$ is irreducible and does not dominate $X$, then $Z$ is a proper reduced irreducible subvariety of $X$. The generic smoothness theorem implies that there is an nonempty open subset $U$ of $X$ such that $U\cap Z$ is smooth. Since $U$ contains the the generic point of $Z$, then $C\cap U_{\infty}$ is an open dense subcylinder of $C$. Note that $U_{\infty}$ is an open subset of $X_{\infty}$, we have $\codim(C,X_{\infty})=\codim (C\cap U_\infty,U_{\infty})$ and $C\subseteq X_{\infty}$ and $C\cap U_{\infty}\subseteq U_{\infty}$ define the same valuation. This implies that we can replace $X$ by $U$ and $C$ by $C\cap U_{\infty}$. As a consequence, we may and will assume that $Z$ is a smooth subvariety of $X$.

If $Z$ is a prime divisor on $X$, then the local ring $\cO_{X,Z}$ is a discrete valuation ring of $k(X)$ with maximal ideal $m_{X,Z}$. Given two local rings $(A,p)$ and $(B,q)$ of $k(X)$, we denote by $(A,p)\preceq (B,q)$ if $A\subseteq B$ is a local inclusion, i.e. $p=q\cap A$. This defines a partial order on the set of local rings of $k(X)$. By the definition of $Z$, we deduce that $$(\cO_{X,Z},m_{X,Z}) \preceq(R,m).$$ Since every valuation ring is maximal with respect to the partial order $\preceq$, it follows that $\cO_{X,Z}$ is equal to the valuation ring $R$ of $\ord_C$, and $\ord_C=q\cdot \ord_Z$ for some integer $q>0$. Therefore we may take $X'=X$ and $E=Z$, in which case $\ord_E(K_{-/X})=0$. The equality $\ord_C{Z}=q\cdot\ord_Z{Z}=q$ implies that $C$ is a subcylinder of ${\rm Cont}^{\geq q}(E)$. Since $E$ is a smooth divisor, we obtain that $\codim {\rm Cont}^{\geq q}(E)=q$. This proves the inequality
$$\codim(C)\geq q\cdot(1+\ord_E(K_{-/X}))=q.$$

We now assume that $Z$ is not a divisor, i.e. $\codim Z\geq 2$.  Let $f: X'\rightarrow X$ be the blow up of $X$ along $Z$. We claim that there exists an irreducible closed cylinder $C'$ in $X'_{\infty}$ such that the morphism $f_{\infty}$ maps $C'$ into $C$ dominantly.

Let $e$ be the vanishing order $\ord_C(K_{X'/X})$. We can assume that $C=(\psi_m^X)^{-1}(S)$ for some closed irreducible subset $S$ in $X_m$ with $m\geq 2e$. The smoothness of $X$ implies that $C\setminus Z_{\infty}$ is a dense subset of $C$. Let $F=f^{-1}(Z)$ be the exceptional divisor on $X'$.
%It follows from the Valuative Criterion for the properness of map $f$ that $X_{\infty}\setminus Z_{\infty}$ is contained in the image of $f_{\infty}$.
It is clear that $f_{\infty}^{-1}(Z_{\infty})=F_{\infty}$.  We denote by
$$\phi: X'_{\infty}\setminus F_{\infty}\rightarrow X_{\infty}\setminus Z_{\infty}$$
the restriction of $f_{\infty}$. Let $\gamma$ be the generic point of $C$ and $L$ the residue field of $\gamma$. Hence $\gamma$ induces a morphism
$$\gamma_L: \Spec L\llb t\rrb\rightarrow X.$$ Lemma \ref{notthin} implies that $\gamma\in X_{\infty}\setminus Z_{\infty}$. By Lemma \ref{proper lift}, we deduce that $\psi$ is bijective on $L$--valued piont, hence there is a unique $L$--valued point of $X'_{\infty}$ mapping to $\gamma_L$ via $\phi$. We denote by $\gamma'$ its underlying point in $X'_{\infty}$. It is clear that $f_{\infty}(\gamma')=\gamma$. For simplicity we write $\gamma_m$ for $\psi_{m}^{X}(\gamma)$ and $\gamma'_m$ for $\psi_{m}^{X'}(\gamma')$. By Lemma \ref{change} part (a), we deduce that $f_m^{-1}(\gamma_m)$ is an affine space of dimension $e$ over the residue field of $\gamma_m$. Hence the image of $f_m^{-1}(\gamma_m)$ in $X'_{\infty}$, denoted by $T$, is irreducible. Since $\gamma_m$ is the generic point of $S$, there is a unique component of $f_m^{-1}(S)$ which contains $T$. Let $S'$ be this component and $C'$ the cylinder $(\psi^{X'}_{m})^{-1}(S')$ in $X'_{\infty}$. We now check that the closed irreducible cylinder $C'$ satisfies the above conditions. The fact $$f_m(\gamma'_m)=f_m\circ\psi^{X'}_{\infty}(\gamma')=\psi^{X}_m(\gamma)=\gamma_m$$
implies that $\gamma'_m\in T$. We deduce that $\gamma'\in C'$. It follows that $f_{\infty}$ maps $C'$ into $C$ dominantly.

The fact that the center of $\ord_C$ on $X$ is $Z$ implies that $\ord_C(F)>0$, hence $e=\ord_{C}(K_{X'/X})>0$. Lemma \ref{change} implies that $f_m: S'\rightarrow S$ is dominant with general fibers of dimensional $e$. We thus have $\dim S'=\dim S+e$, hence
\begin{equation*}
\codim C'=\dim X'_m-\dim S'=\dim X_m-(\dim S+e)=\codim C-e
\end{equation*}

We now set $X^{(0)}=X, X^{(1)}=X',C^{(0)}=C$ and $C^{(1)}=C'$. By the construction of $C'$, we deduce that $\ord_C$ and $\ord_{C'}$ are equal as valuations of $k(X)$. If the center of $\ord_{C'}$ on $X'$ is not a divisor, we blow up this center again (we may need to shrink $X'$ to make the center to be smooth). We now run the above argument for the variety $X^{(1)}$ and $C^{(1)}$ and obtain $X^{(2)}$ and $C^{(2)}$. Since every such blow up decreases the codimension of the cylinder, which is an non-negative integer, we deduce that after $s$ blow ups, the center of the valuation $\ord_{C^{(s)}}$ on $X^{(s)}$ is a divisor, denoted by $E$. We have $$\ord_C=\ord_{C^{(1)}}=\cdots=\ord_{C^{(s)}}=q\cdot\ord_E.$$ We now check the inequality $\codim C\geq q\cdot(1+\ord_E(K_{-/X}))$. At each step, we have
\begin{equation*}
\begin{array}{cl}
\codim(C)=\codim (C^{(1)})+\ord_C(K_{X^{(1)}/X})\\
\codim(C^{(1)})=\codim (C^{(2)})+\ord_C(K_{X^{(2)}/{X^{(1)}}})\\
\cdots \\
\codim(C^{(s-1)})=\codim (C^{(s)})+\ord_C(K_{X^{(s)}/{X^{(s-1)}}})
\end{array}
\end{equation*}

We thus obtain that
\begin{equation*}
\begin{array}{cl}
\codim (C)&=\codim (C^{(s)})+\sum\limits_{i=1}^{s} \ord_C(K_{X^{(i)}/{X^{(i-1)}}})\\
          &=\codim (C^{(s)})+\ord_C(K_{X^{(s)}/{X}})
\end{array}
\end{equation*}
It is clear that $\ord_C(E)=q \cdot \ord_{E}(E)=q$, hence $C^{(s)}\subseteq {\rm Cont}^{\geq q}(E)$, and therefore $\codim C^{(s)}\geq \codim {\rm Cont}^{\geq q}(E)=q$. This complete the proof.
\end{proof}

\begin{lemma}\label{closure of cylinder}
Let $X$ be a smooth variety and $S$ a constructible subset of $X_m$ for some $m$.
\begin{itemize}
\item[(a)] $\overline{\psi^{-1}_m(S)}=\psi^{-1}_m(\overline{S})$.
\item[(b)] If $U$ is an open subset of $X$ and $C$ is a cylinder in $U_{\infty}$, then the closure $\overline{C}$ in $X_{\infty}$ is a closed cylinder in $X_{\infty}$.
\end{itemize}
\end{lemma}

\begin{proof} We first prove part (a). Since $\psi_m$ is continuous with respect to the Zariski topologies, we deduce that $\psi^{-1}_m(\overline{S})$ is closed. We thus have $\overline{\psi^{-1}_m(S)}\subseteq\psi^{-1}_m(\overline{S})$. If $\overline{\psi^{-1}_m(S)}\neq\psi^{-1}_m(\overline{S})$, then there is an arc $\gamma \in \psi^{-1}_m(\overline{S})\setminus \overline{\psi^{-1}_m(S)}$. Let $U$ be an affine neighborhood of $\psi_0(\gamma)$ in $X$ and $W=S\cap U_m$. It is clear that $$\gamma \in (\psi^U_m)^{-1}(\overline{W})\setminus\overline{(\psi^U_m)^{-1}(W)}.$$ In order to get a contradiction, we can replace $X$ by $U$ and $S$ by $W$. We thus may assume that $X$ is an affine variety. It follows from the construction of jet schemes that $X_m$ are smooth affine varieties. Let $X_m=\Spec A_m$ for every $m\geq 0$. Hence $X_{\infty}=\Spec A$ where $A=\bigcup\limits_m A_m$. We claim that if $\overline{\psi^{-1}_m(S)}\neq\psi^{-1}_m(\overline{S})$, then there is an integer $n\geq m$ such that $$\overline{\psi_{n}(\overline{\psi^{-1}_m(S)})}\neq \psi_n(\psi^{-1}_m(\overline{S})).$$
Since $\psi_n(\psi^{-1}_m(\overline{S}))=(\rho^{n}_m)^{-1}(\overline{S})$ and $\psi_n(\psi^{-1}_m({S}))=(\rho^{n}_m)^{-1}({S})$, we deduce that $$\overline{(\rho^{n}_m)^{-1}({S})}=\overline{\psi_{n}(\psi^{-1}_m(S))}\subseteq \overline{\psi_{n}(\overline{\psi^{-1}_m(S)})}\subsetneq (\rho^{n}_m)^{-1}(\overline{S}).$$
On the other hand, since $\rho^n_m$ is a locally trivial affine bundle with fiber $\AAA^{\dim X(n-m)}$, we have $\overline{(\rho^{n}_m)^{-1}({S})}=(\rho^{n}_m)^{-1}(\overline{S})$. We thus get an contraction.

We now prove the claim. Let $I$ be the radical ideal defining $\overline{\psi^{-1}_m(S)}$ in $X_\infty$ and $J$ the radical ideal defining $\psi^{-1}_m(\overline{S})$. If $\overline{\psi^{-1}_m(S)}\neq\psi^{-1}_m(\overline{S})$, then there is an element $f\in I\setminus J$. There exist an integer $n\geq m$ such that $f\in A_n$. Let $I_n=I\cap A_n$ and $J_n=J\cap A_n$. It is clear that $\overline{\psi_n(\psi^{-1}_m(\overline{S}))}$ is the closed subset of $X_n$ defined by $J_n$. Similarly $\psi_n(\overline{\psi^{-1}_m(S)})=(\rho^n_m)^{-1}(\overline{S})$ is the closed subset of $X_n$ defined by the ideal $I_n$. Since $f\in I_n\setminus J_n$, we thus have the assertion of the claim. This completes the proof of part (a).

For the proof of part (b), let $C=(\psi^U_{m})^{-1}(S)$ for some integer $m\geq 0$ and some constructible subset $S$ of $U_m$. We now consider $S$ as a constructible subset of $X_m$ and apply part (a), we thus obtain $\overline{C}=\overline{\psi^{-1}_m(S)}=(\psi^X_m)^{-1}(\overline{S})$. This completes the proof.
\end{proof}

\begin{lemma}\label{bl}
Let $X$ and $X'$ be smooth varieties over a field $k$, and $f: X'\rightarrow X$ a blow up with smooth center. If $C'$ is a closed cylinder of $X'$, then the closure of the image $f_{\infty}(C')$, denoted by $C$, is a cylinder in $X'$. We also have
$$\ord_C=\ord_{C'};~\codim C=\codim C'+\ord_{C'}{K_{X'/X}}.$$
\end{lemma}

\begin{proof}
Let $e=\ord_{C'}{K_{X'/X}}$.  For simplicity, we write $\psi'_m$ for $\psi^{X'}_m$ and $\psi_m$ for $\psi^X_m$ for every $m\geq 0$. We first show that $C$ is a closed cylinder. We choose an integer $p\geq e$ and a constructible subset $T'$ of $X'_p$ such that $C'=(\psi'_p)^{-1}(T')$. Let $m=e+p$. We denote by $S'$ the inverse image of $T'$ by the canonical projection $\rho^{m}_p: {X'}_m\rightarrow {X'}_p$. Let $S=f_m(S')$. Lemma \ref{change} part (b) implies that $f_m^{-1}(f_m(S'))\subseteq (\rho^m_{p})^{-1}(T')=S'$. We thus have $f_m^{-1}(f_m(S'))=S'$. It follows that $f_{\infty}(C')=\psi^{-1}_m(S)$. Hence
$$C=\overline{f_{\infty}(C')}=\overline{\psi^{-1}_m(S)}=\psi^{-1}_m(\overline{S})$$
is an irreducible closed cylinder in $X_{\infty}$. Here the last equality follows from Lemma \ref{closure of cylinder} part (a).
Since $C'$ dominates $C$, we have $\ord_C=\ord_{C'}$. The codimension equality follows from the fact that $\dim S'=\dim S+e$ by Lemma \ref{change}.
\end{proof}

\begin{lemma}\label{valtocyl}
Let $X$ be a smooth variety over a perfect field $k$. If $f: Y\rightarrow X$ is a birational morphism from a normal variety $Y$ and $E$ is a prime divisor, then for every positive integer $q$, there exist an irreducible cylinder $C\subset X_{\infty}$ such that $\ord_C=q\cdot\ord_E$ and
\begin{equation}\label{codim}
\codim(C)=q\cdot(1+\ord_E(K_{Y/X}))
\end{equation}
\end{lemma}

\begin{proof} Let $\nu$ be the divisorial valuation $q\cdot\ord_E$ on the function field $K(X)$. We define a sequence of varieties and maps as follows. Let $Z^{(0)}$ be the center of $\nu$ on $X$ and $X^{(0)}=X$. We choose an open subset $U^{(0)}$ of $X^{(0)}$ such that $Z^{(0)}\cap U^{(0)}$ is a nonempty smooth subvariety of $U^{(0)}$. If $Z^{(0)}\cap U^{(0)}$ is not a divisor, then let $f_1: X^{(1)}\rightarrow U^{(0)}$ be the blow up of $U^{(0)}$ along $Z^{(0)}\cap U^{(0)}$ and $h_1:X^{(1)}\rightarrow X$ the composition of $f_1$ with the embedding $U^{(0)}\hookrightarrow X$. If $f_i: X^{(i)}\rightarrow U^{(i-1)}$ and $h_i: X^{(i)}\rightarrow X^{(i-1)}$ are already defined, then we denote by $Z^{(i)}$ the center of $\nu$ on $X^{(i)}$. We pick an open subset $U^{(i)}\subset X^{(i)}$ such that $Z^{(i)}\cap U^{(i)}$ is a smooth subvariety of $U^{(i)}$. If $Z^{(i)}$ is not a divisor, then we denote by $f_{i+1}: X^{(i+1)}\rightarrow U^{(i)}$ the blow up of $U_i$ along $Z^{(i)}\cap U^{(i)}$ and $h_{i+1}: X^{(i+1)}\rightarrow X^{(i)}$ the composition of $f_{i+1}$ with the embedding $U^{(i)}\rightarrow X^{(i)}$. By \cite{KM}[Lemma 2.45], we know there is an integer $s\geq 0$ such that $Z^{(s)}$ is a prime divisor on $U^{(s)}$ and $\ord_{Z^{(s)}}=\ord_E$. Hence we can replace $Y$ by a smooth variety $U^{(s)}$ and $E=Z^{(s)}\cap U^{(s)}$. We write $g_i: Y\rightarrow X^{(i)}$ for the composition of morphisms $h_j$ for $j$ with $i<j\leq s$ and the embedding $U^{(s)}\subset X^{(s)}$.

Let $C_s$ be the locally closed cylinder ${\rm Cont}^{q}(E)$ in $Y_{\infty}$ and $C_0$ the closure of its image $(g_0)_{\infty}(C_s)$ in $X_{\infty}$. It is clear that $\codim C_s=q$. We now show that $C=C_0$ is a cylinder that satisfies our conditions. For every $i$ with $1\leq i\leq s$, we denote by $C_i$ the closure of the image of $C_s$ in $X^{(i)}_{\infty}$ under the map $(g_{i})_{\infty}: Y_{\infty}\rightarrow X^{(i)}_{\infty}$. Similarly, we denote by $D_i$ the closure of the image of $C_s$ in $U^{(i)}_{\infty}$. It is clear that $D_i$ is the closure of the image of $C_{i+1}$ in $U^{(i)}_{\infty}$ under the map $(f_{i+1})_{\infty}: X^{(i+1)}_{\infty}\rightarrow U^{(i)}_{\infty}$ and $C_i$ is the closure of $D_i$ in $X^{(i)}_{\infty}$. By Lemma \ref{bl} and Lemma \ref{closure of cylinder} part (b), using descending induction on $i<s$, we deduce that $D_i$ is a cylinder in $U^{(i)}_{\infty}$ and $C_i$ is a cylinder in $(X^{(i)})_{\infty}$.  We also deduce that $\ord_{C_i}=\ord_{D_i}=\ord_{C_{i+1}}$ and $$\codim C_i=\codim D_i=\codim C_{i+1}+ \ord_{C_i}{K_{X^{(i+1)}/X^{(i)}}}.$$

We thus obtain $\ord_C=\ord_{C_{1}}=\cdots =\ord_{C_{s}}=q\cdot\ord_E$ and
\begin{equation*}
\begin{array}{cl}
\codim C&=\codim C_1+ \ord_C(K_{X^{(1)}/X})\\
&\cdots\\
&=\codim C_s+\sum\limits_{i=0}^{s-1} \ord_C(K_{X^{(i+1)}/X^{(i)}})=q+q\cdot\ord_E(K_{Y/X}).
\end{array}
\end{equation*}
\end{proof}

It is clear that Theorem \ref{corr} follows from Lemma \ref{cyltoval} and Lemma \ref{valtocyl}.
We now prove Theorem \ref{jetformula}.

\begin{proof}
If $Y=X$, the assertion is trivial. Hence we may and will assume $Y$ is a closed subscheme of $X$ and $Y\neq X$.
By Theorem \ref{corr}, we deduce that
\begin{equation*}
\lct(X,Y):=\inf\limits_{E}\frac{1+\ord_E(K_{-/X})}{\ord_E(Y)}=\inf\limits_{C}\frac{\codim C}{\ord_C(Y)}.
\end{equation*}
where $C$ varies over the irreducible closed cylinders which do not dominate $X$.

We first show that
\begin{equation}\label{star}\tag{$\dag$}
\lct(X,Y)\leq\inf\limits_{m\geq 0}\frac{\codim(Y_m,X_m)}{m+1}
\end{equation}
For every $m\geq 0$, let $S_m$ be an irreducible component of $Y_m$ which computes the codimension of $Y_m$ in $X_m$ and $C_m$ the closed irreducible cylinder $\psi^{-1}_m(S_m)$ in $X_{\infty}$. We thus obtain $$\codim(C_m)=\codim(S_m,X_{m})=\codim(Y_m,X_m).$$
The image $\psi_0(C_m)=\rho^m_0(S_m)$ is contained in $Y$, which implies that $C_m$ does not dominate $X$. By the definition of contact loci, we know that $Y_m={\rm Cont}^{\geq m+1}(Y)_m$ in $X_m$. This implies that $\ord_{C_m}(Y)\geq  m+1$. We conclude that
$$\lct(X,Y)\leq\frac{\codim(C_m)}{\ord_{C_m}(Y)}=\frac{\codim(Y_m,X_m)}{m+1}.$$ Taking infimum over all integers $m\geq 0$, we now have the inequality (\ref{star}).

We now prove the reverse inequality.
Given an irreducible closed cylinder $C$ which does not dominate $X$. If $\ord_C(Y)=0$, then $\frac{\codim C}{\ord_C(Y)}=\infty$. Hence $$\frac{\codim C}{\ord_C(Y)}\geq \inf\limits_{m\geq 0}\frac{\codim(Y_m,X_m)}{m+1}.$$ From now on, we may and will assume that $\ord_C(Y)>0$. Let $m=\ord_C(Y)-1$. Since $C$ is a subcylinder of the contact locus ${\rm Cont}^{\geq m+1}(Y)=\psi_m^{-1}(Y_m)$, we have
$$\frac{\codim C}{\ord_C(Y)}\geq \frac{\codim(Y_m,X_m)}{m+1}\geq \inf\limits_{m\geq 0}\frac{\codim(Y_m,X_m)}{m+1}.$$ We now take infimum over all cylinders $C$ which do not dominate $X$ and obtain
$$\lct(X,Y)=\inf_{C}\frac{\codim C}{\ord_C(Y)}\geq  \inf\limits_{m\geq 0}\frac{\codim(Y_m,X_m)}{m+1}.$$
\end{proof}

Let $X$ be a smooth variety over a perfect field $k$, $Y$ a closed subscheme of $X$ and $Z$ a closed subset of $X$. Recall that
$$\lct_Z(X,Y)=\inf\limits_{E/X}\frac{\ord_E(K_{-/X})+1}{\ord_E{Y}}$$
where $E$ varies over all divisors over $X$ whose center in $X$ intersects $Z$. By the correspondence in Theorem \ref{corr}(2), we deduce that for every such divisor $E$ over $X$, the corresponding closed irreducible cylinder $C$ satisfies $$\overline{\psi_0^X(C)}\cap Z\neq \emptyset.$$ Applying the argument in the proof of Theorem \ref{jetformula}, we can show the following generalized log canonical threshold formula in terms of jet schemes.

\begin{proposition}\label{genformula}
Let $(X,Y)$ be a pair over a perfect field $k$ and $Z$ a closed subset of $X$. We have
$$\lct_Z(X,Y)=\inf\limits_{C\subset X_{\infty}}\frac{\codim C}{\ord_C(\fra)}=\inf\limits_{m\geq 0}\frac{\codim_Z(Y_m,X_m)}{m+1}$$
where $C$ varies over all irreducible closed cylinders with $\overline{\psi_0(C)}\cap Z\neq\emptyset, \overline{\psi_0(C)}\neq X$ and $\codim_Z(Y_m,X_m)$ is the minimum codimension of an irreducible component $T$ of $Y_m$ such that $\overline{\pi_m(T)}\cap Z\neq \emptyset$.
\end{proposition}

%\begin{remark}
%If $U$ is an open neighborhood of $Z$ in $X$ such that for every $T$ with $\pi_m(T)\cap Z=\emptyset$, $U\cap \pi_m(T)=\emptyset$. In particular, this implies that $\dim_Z(Y_m)=\dim(Y\cap U)_m$. ****question**** do we also have $\lct_Z(X,Y)=\lct(U,Y\cap U)$?
%\end{remark}

\begin{remark} We have seen that  $$\lct(X,Y):=\inf\limits_{E}\frac{1+\ord_E(K_{-/X})}{\ord_E(Y)}=\inf\limits_{C}\frac{\codim C}{\ord_C(Y)}=\inf\limits_{m\geq 0}\frac{\codim(Y_m,X_m)}{m+1}.$$
If one of the infumums can be achieved, then so are the other two.  In particular, when the base field $k$ is of characteristic $0$, log resolutions of $(X,Y)$ exist. Hence the log canonical threshold $\lct(X,Y)$ can be computed at some exceptional divisor $E$ in the log resolution. In this case, all the infimums can be replaced by minimuns.
\end{remark}

\begin{remark} Let $k$ be an algebraically closed field of characteristic $0$ and $K=k(s)$ the function field of $\AAA^1_k$. Hence $K$ is not a perfect field. There are examples of pairs $(X,Z)$ over $K$ such that the formula in Theroem \ref{jetformula} does not hold. For instance, let $X=\Spec K[x]$ and $Y$ be a prime divisor on $X$ defined by a single equation $(x^p-s)$.  It is easy to check that $\lct(X,Y)=1$. On the other hand, for every $m$, $X_m=\AAA^{m+1}$ and $Y_m=\AAA^{m-\lfloor \frac{m}{p}\rfloor}$. Hence $\inf\limits_{m\geq 0}\frac{\codim(Y_m,X_m)}{m+1}=1/p$. We thus have $\lct(X,Y)\neq \inf\limits_{m\geq 0}\frac{\codim(Y_m,X_m)}{m+1}$.
\end{remark}

\section{The log canonical threshold via jets}
In this section, we apply Theorem \ref{jetformula} to deduce properties of log canonical threshold for pairs.
Our first corollary of Theorem \ref{jetformula} is the following comparison result in the setting of reduction to prime characteristic. Suppose that $X$ is the affine variety $\AAA^n_{\ZZ}$ over the ring $\ZZ$ and $Y$ is a subscheme of $X$ defined by an ideal $\fra\subset \ZZ[x_1,\cdots,x_n]$ that contained in the ideal $(x_1, \cdots,x_n)$. For every prime number $p$, let $X_p=\AAA^n_{\FF_p}$ and $Y_p$ be the subscheme of $X_p$ defined by $\fra\cdot \FF_p[x_1,\cdots,x_n]$. Note that a log resolution of $(X_{\QQ},Y_{\QQ})$ induces a log resolution of the pair $(X_p,Y_p)$ for $p$ large enough. It follows that $\lct_0(Y_{\QQ},X_{\QQ})=\lct_0(Y_{p},X_{p})$ for all but finitely many $p$. We now prove the following inequality for every prime $p$.

\begin{corollary}\label{comp0andp}
If $(X,Y)$ is a pair as above, then for every prime integer $p$, we have
$$\lct_0(X_{\QQ},Y_{\QQ})\geq \lct_0(X_{p},Y_{p}),$$
where the log canonical thresholds are computed at the origin.
\end{corollary}

\begin{proof}
Using \cite{Mmus2}[Corollary 3.6], we obtain
$$\lct_0(Y_{\QQ},X_{\QQ})= \inf\limits_{m\geq 0}\frac{\codim((Y_{\QQ})_{m,0},(X_{\QQ})_m)}{m+1}.$$
By Proposition \ref{genformula}, for every integer $m\geq 0$, we have
$$\lct_0(X_p,Y_p)\leq \frac{\codim_0((Y_{p})_m,(X_p)_m)}{m+1}\leq \frac{\codim((Y_{p})_{m,0},(X_p)_m)}{m+1}.$$
In order to complete the proof, it is enough to show that for every $m\geq 1$ and every prime $p$,
$$\codim((Y_{p})_{m,0},(X_p)_m)\leq \codim((Y_{\QQ})_{m,0},(X_{\QQ})_m).$$ Since $\dim(X_{\QQ})_m=\dim(X_p)_m=n(m+1)$, it suffices to show that $$\dim(Y_p)_{m,0}\geq \dim(Y_{\QQ})_{m,0}$$ for every $p$.

Let $S$ be $\Spec\ZZ$. Recall that $(Y/S)_m$ is the $m^{\text{th}}$ relative jet scheme of $Y/S$. Let $\tau: \Spec \ZZ\rightarrow Y$ be the zero section.
Since $\fra\subset (x_1, \cdots, x_n)$, the map $\tau$ factors through $Y$. By Lemma \ref{semicont}, we deduce that for every $m\geq 1$, the function $$f(s)=\dim(Y_s)_{m,\tau(s)}=\dim(Y_s)_{m,0}$$ is upper semi-continuous on $S$. Hence we have $\dim(Y_p)_{m,0}\geq \dim(Y_{\QQ})_{m,0}$ for every $m$ and $p$.
This completes the proof.
\end{proof}

This in turn has an application to an open problem about the connection between
log canonical thresholds and $F$-pure thresholds. Recall that in positive characteristic
Takagi and Watanabe \cite{TW} introduced an analogue of the log canonical threshold,
the $F$-pure threshold. With the above notation, it follows from
\cite{HW} that $\lct_0(X_p,Y_p)\geq \fpt_0(X_p,Y_p)$ for every prime $p$, where $\fpt_0(X_p,Y_p)$ is the $F$--pure threshold of the pair $(X_p,Y_p)$ at $0$. By combining this
with Corollary~\ref{comp0andp}, we obtain the following result, which seems to have been
an open question.

\begin{corollary}\label{hycomp}
With the above notation, we have $\lct_0(X_{\QQ}, Y_{\QQ}))\geq \fpt_0(X_p,Y_p)$
for every prime $p$.
\end{corollary}

Let $k$ be a perfect field and $\overline{k}$ be the algebraic closure of $k$. For every scheme $X$ over $k$, we denote by $\overline{X}$ the fiber product $X\times_k\Spec \overline{k}$.

\begin{corollary}\label{passtoclosure}
Let $X$ be a smooth variety over a perfect field $k$ and $Y$ a closed subscheme of $X$. We have
$$\lct(X,Y)=\lct(\overline{X},\overline{Y}).$$
\end{corollary}

\begin{proof}
For every scheme $Z$ over field $k$, we know that $\dim Z=\dim \overline{Z}$. We thus have for every $m\geq 0$,
$$\codim(Y_m, X_m)=\codim(\overline{Y}_m,\overline{X}_m).$$ Our assertion follows from Theorem \ref{jetformula}.
\end{proof}

\begin{remark}
Corollary \ref{passtoclosure} is not true if the base field is not perfect. For instance, let $k$ be an algebraically closed field and $K=k(s)$ the function field of $\AAA^1_k$. Let $X=\Spec K[x]$ and $Y$ be the closed subscheme of $X$ defined by $x^p-s$. We have seen that $\lct(X,Y)=1$.
Let $\overline{K}$ be the algebraic closure of $K$. We thus have $X_{\overline K}=\AAA^1_{\overline{K}}$ and $Y_{\overline{K}}$ is a nonreduced subscheme of $X_{\overline{K}}$ defined by $(x-s^{1/p})^p$. One can check that $\lct(X_{\overline{K}},Y_{\overline{K}})=1/p$.
\end{remark}

\begin{corollary}\label{IOA}
Let $X$ be a smooth variety over a perfect field $k$ and $Y$ a closed subscheme of $X$. If $H$ is a smooth irreducible divisor on $X$ which intersects $Y$ and $Z\subset H$ is a nonempty closed subset, then
$$\lct_{Z}(X,Y)\geq \lct_{Z}(H,H\cap Y).$$
\end{corollary}

\begin{proof} The case $H\cap Y=H$ is trivial since $\lct_Z(H,H\cap Y)=0$. We may thus assume $Y\cap H\neq H$. Similarly, if $Z\cap Y=\emptyset$, then both $\lct_{Z}(X,Y)$ and $\lct_Z(H,H\cap Y)$ are equal to $\infty$. We should assume $Z\cap Y\neq \emptyset$ from now on.

By Proposition \ref{genformula}, we only have to prove that for every $m\geq 0$,
$$\codim_Z(Y_m,X_m)\geq \codim_Z((H\cap Y)_m,H_m).$$
Let $T$ be an irreducible component of $Y_m$ such that
\begin{equation*}
\pi_m(T)\cap Z\neq \emptyset ~\text{and}~ \codim T=\codim_Z(Y_m,X_m).
\end{equation*}
Since $H$ is a Cartier divisor on $X$, $H\cap Y$ is defined locally in Y by one equation. This implies that $(H\cap Y)_m=H_m\cap Y_m$ is defined locally in $Y_m$ by $m+1$ equations. If $$\pi_m(T\cap H_m)\cap Z\neq \emptyset,$$ then there is a component of $T\cap H_m$, denoted by $S$, such that $\pi_m(S)\cap Z\neq \emptyset $ and $\dim S\geq \dim T-(m+1)$. Note that $\dim X_m=\dim H_m +m+1$ and we conclude that
$$\codim_Z((H\cap Y)_m,H_m)\leq \codim(S,H_m)\leq \codim(T\cap H_m,H_m)\leq\codim(T,X_m).$$

We now prove that $\pi_m(T\cap H_m)\cap Z\neq \emptyset$. Let $\gamma_m\in T$ such that $\pi_m(\gamma_m)\in Z$. Recall that $\sigma_m: Y\rightarrow Y_m$ is the zero section.  Since $T$ is invariant under the action of $\AAA^1$, the orbit of $\gamma_m$ is a subset of $T$. In particular, $\sigma_m(\pi_m(\gamma_m))\in T$. %Hence we have $\sigma_m(\pi_m(T))\subset T$.
Since the zero section is functorial by its construction, we get $\sigma_m(Y\cap H)\subset Y_m\cap H_m$.  In particular, $\sigma_m(\pi_m(\gamma_m))$ is in $T\cap H_m$ and its image under $\pi_m$ is in $Z$. This completes our proof.
\end{proof}

\begin{corollary}
If $X$ is a smooth complex variety and $Y\subset X$ is a proper closed subscheme, then for we have
$\lct(X,Y)>0$.
\end{corollary}

\begin{proof}
Since log canonical thresholds can be computed after passing to an algebraic closure of $k$, we can assume $k$ is algebraically closed. It follows from the definition that $$\lct(X,Y)=\inf\limits_{x\in Y}\lct_x(X,Y).$$ For every $x\in Y$, we will show that
\begin{equation}\label{lctmult}
\lct_x(X,Y)\geq 1/{\ord_x(Y)}.
\end{equation}
We thus have $\lct_x(X,Y)\geq 1/d$ where $d=\max\limits_{x\in Y} \,\ord_x(Y)$. Here $\ord_xY$ is the maximal value $q$ such that $I_{Y,x}\subseteq m^q_{X,x}$, where $m_{X,x}$ is the ideal defining $x$.

We prove the inequality (\ref{lctmult}) by induction on $\dim(X)$. If $X$ is a smooth curve, then it follows from definition that $\lct_{x}(X,Y)=\ord_xY$. We now assume that $\dim X\geq 2$. After replacing $X$ by an open neighborhood of $x$, we may find $H$, a smooth divisor passing through $x$, such that $\ord_x(H\cap Y)=\ord_xY$. By Corollary \ref{IOA}, we have $$\lct_x(X,Y)\geq \lct_x(H,H\cap Y)\geq 1/{\ord_x(H\cap Y)}=1/\ord_xY.$$
This completes the proof.
\end{proof}

\providecommand{\bysame}{\leavevmode \hbox \o3em
{\hrulefill}\thinspace}


\begin{thebibliography}{HY}

\bibitem[Bli]{Bli} M.~Blickle, A short course on geometric motivic integration, Motivic integration and its interactions with model theory and non-Archimedean geometry. Volume I, 189¨C243, London Math. Soc. Lecture Note Ser., \textbf{383}, Cambridge Univ. Press, Cambridge, (2011).

\bibitem[DL]{DL} J.~Denef and F.~Loeser, Germs of arcs on singular algebraic varieties an motivic integration, Invent. Math. \textbf{135} (1999), 201--232.

\bibitem[Eis]{Eis} D.~Eisenbud, Commutative Algebra with a View Torward Algebraic Geoemtry, Graduate Texts in Mathematics \textbf{150}, Springer Verlag, (1995).

\bibitem[ELM]{ELM} L.~Ein, R.~Lazarsfeld, and M.~Musta\c{t}\v{a}, Contact loci in arc spaces, Compos. Math. \textbf{140} (2004), 1229--1244.

\bibitem[EM]{EM}
L.~Ein, M.~Musta\c{t}\v{a}, Jet schemes and singularities, Algebraic geometry--Seattle 2005, Part \textbf{2}, 505--546, Proc. Sympos. Pure Math., \textbf{80}, Part \textbf{2}, Amer. Math. Soc., Providence, RI, 2009.

\bibitem[EMY]{EMY}
L.~Ein, M.~Musta\c{t}\v{a} and T.~Yasuda, Jet schemes, log discrepancies and inversion of adjunction, Invent. Math. \textbf{153} (2003), 519--535.

\bibitem[FEI]{FEI} T.~ De Fernex, L.~Ein and S.~Ishii, Divisorial valuation via arcs, Publ. RIMS \textbf{44} (2008), 425--448.

\bibitem[HW]{HW} N.~Hara and K.-i.~Watanabe,
$F$-regular and $F$-pure rings vs. log terminal and log canonical singularities,
 J. Algebraic Geom. \textbf{11} (2002), 363--392.

\bibitem[KM]{KM} J.~Koll\'{a}r and S.~Mori, Birational Geometry of Algebraic Varieties, Cambridge University Press, 1998.

\bibitem[Kol]{Kol} J.~Koll\'{a}r, Singularities of pairs, Algebraic geometry¡ªSanta Cruz 1995, Proc. Sympos. Pure Math., vol. \textbf{62}, Amer. Math. Soc., Providence, RI, 1997, pp. 221¨C287.

\bibitem[Kon]{Kon}
M.~Kontsevich, Motivic integration. Lecture at Orsay, 1995.

\bibitem[Ish]{Ish}
S.~Ishii, Arcs, valuations and the Nash map, J. reine angew. Math. \textbf{588} (2005), 71-92.

\bibitem[Mus1]{Mus1}
M.~Musta\c{t}\v{a}, Jet schemes of locally complete intersection
canonical singularities, with an appendix by
David Eisenbud and Edward Frenkel, Invent. Math., \textbf{145} (2001), 397--424.

\bibitem[Mus2]{Mmus2} M.~Musta\c{t}\v{a}, Singularities of pairs via jet schemes,  J. Amer. Math. Soc. \textbf{15} (2002), 599--615.

\bibitem[TW]{TW}
S.~Takagi and K.~Watanabe, On F-pure thresholds, J. Algebra \textbf{282} (2004), 278--297.

\end{thebibliography}
\end{document}